\documentclass[amssymb,twoside,11pt]{article}
\usepackage{latexsym,amsmath,graphicx}
\usepackage{amsfonts}
\newtheorem{theorem}{Theorem}[section]
\newtheorem{lemma}[theorem]{{\bf Lemma}}

\newtheorem{cor}[theorem]{{\bf Corollary}}
\newtheorem{rem}[theorem]{{\bf Remark}}
\newtheorem{ex}[theorem]{{\bf Example}}
\newtheorem{definition}{Definition}[section]
\numberwithin{equation}{section}
\newenvironment{proof}{\indent{\em Proof:}}{\quad \hfill
$\Box$\vspace*{2ex}}

\font\Bbb=msbm10 at 12pt

\newcommand{\N}{\mbox{\Bbb N}}
\begin{document}
\setcounter{page}{1}
\begin{center}
\vspace{0.4cm} {\large{\bf Global Existence and Ulam--Hyers Stability  of
$\Psi$--Hilfer  Fractional Differential Equations}}\\\vspace{0.5cm}
Kishor D. Kucche $^{1}$ \\
kdkucche@gmail.com \\
\vspace{0.35cm}
Jyoti P. Kharade  $^{2}$\\
jyoti.thorwe@gmail.com\\
\vspace{0.35cm}
$^{1,2}$ Department of Mathematics, Shivaji University, Kolhapur-416 004, Maharashtra, India.
\end{center}
\def\baselinestretch{1.0}\small\normalsize
\begin{abstract}
In this paper, we consider  the Cauchy-type problem for a nonlinear differential equation involving $\Psi$-Hilfer fractional derivative and prove the existence and uniqueness of solutions in the  weighted space of functions. The Ulam--Hyers and Ulam--Hyers--Rassias stability of Cauchy--type problem is investigated via successive approximation method. Further, we investigate the dependence of solutions on the initial conditions  and uniqueness  via $\epsilon$-approximated solution. An example is provided to illustrate  the results we obtained.
\end{abstract}
\def\baselinestretch{1.5}
\allowdisplaybreaks
\section{Introduction}
The theory and applications of fractional differential equations (FDEs) \cite{Kilbas} is the field of great interest  in pure mathematics and in applications. Many applications of FDEs in various branches of science and engineering intended researchers to work on different kinds of problems  related to  this field. There are many different definitions of fractional derivative and fractional integral which do not coincide in general.
Hilfer \cite{Hilfer} introduced the generalized Riemann--Liouville fractional derivative of order $\mu  ~(n-1<\mu<n \in \mathbb{N})$  and of type $\nu ~(0\leq \nu \leq 1)$, defined by
\begin{equation*}
\mathbf{D}_{a+}^{\mu ,\nu }y\left( t\right) =\mathbf{I}_{a+}^{\nu \left(
n-\mu \right)}\left(\frac{d}{dx}\right) ^{n}\mathbf{I}_{a+}^{\left( 1-\nu \right) \left( n-\mu
\right) }y\left( t\right),
\end{equation*}
which allows to interpolate between the  Riemann-Liouville derivative $\mathbf{D}_{a+}^{\mu ,0 }=~ ^{RL}\mathbf{D}_{a+}^{\mu }$ and the Caputo derivative  $\mathbf{D}_{a+}^{\mu ,1 }= ~^{C}\mathbf{D}_{a+}^{\mu }$. Furati et al. \cite{Furati,Furati1} have dealt with the basic problems of existence, uniqueness and stability  of nonlinear Cauchy type problem involving  Hilfer fractional derivative.

In very recent contribution, Sousa and Olivera \cite{Sousa1} extended the concept of Hilfer derivative operator and introduced a new definition of fractional derivative namely $\Psi$--Hilfer fractional derivative of a function of order $\mu$ and of type $\nu$ with respect to another function $\Psi$, discussed its calculus and derived a class of fractional integrals and fractional derivatives by giving a particular value to the function $\Psi$.  In \cite{Sousa} Sousa and Olivera proved generalized Gronwall inequality involving fractional integral with respect to the another function and investigated basic results pertaining to existence, uniqueness  and data dependence of Cauchy type problem involving $\Psi$-Hilfer differential operator.

The fundamental problem of Ulam \cite{Ulam} was generalized for the stability of FDEs \cite{Wang}. Stability of any FDE in Ulam--Hyers sense is the problem of  dealing with the replacement of given FDE by a fractional differential inequality and to obtain the sufficient conditions about ``When the solutions of the fractional differential inequalities are close to the solutions of given FDE ?". For Ulam--Hyers  stability theory of fractional differential equations  and its recent development, one can refer to \cite{Wang,Abbas,Abbas1,Sousa2,Benchohra} and the references therein. 

Huang et al. \cite{Huang} investigated HU stability of integer order delay differential equations by method of successive approximation. Kucche and Sutar \cite{Kucche} have  extended the idea of \cite{Huang}  and investigated the HU stability of nonlinear delay FDEs with Caputo derivative.

Motivated by the work of \cite{Sousa1,Sousa,Kucche}, in this paper, we consider the $\Psi$--Hilfer fractional differential equation ($\Psi$--Hilfer FDE) of the form:
\begin{align}
^H \mathbf{D}^{\mu,\nu;\,\Psi}_{a^+}y(t)&=f(t, y(t)),~t \in [a,b], ~0<\mu<1, ~0\leq\nu\leq 1,\label{e11}\\
\mathbf{I}_{a^+}^{1-\rho;\, \Psi}y(a)&=y_a \in \mathbb{R}, ~\rho=\mu+\nu-\mu\nu, \label{e12}
\end{align}
where $^H \mathbf{D}^{\mu,\nu;\,\Psi}_{a^+}(\cdot)$ is the (left-sided) $\Psi$-Hilfer fractional derivative of order $\mu$ and type $\nu$, $\mathbf{I}_{a^+}^{1-\rho;\, \Psi}$ is (left-sided) fractional integral of order $1-\rho$ with respect to another function $\Psi$ in Riemann--Liouville sense and $f: [a,b] \times \mathbb{R}  \to \mathbb{R} $ is a given function that will be specified latter.

The main objective of this paper is to prove the global existence and uniqueness of solution 
to  $\Psi$--Hilfer fractional differential equations \eqref{e11}-\eqref{e12}. Using method of successive approximations we investigate the Ulam--Hyers (HU) and Ulam--Hyers--Rassias (HUR) stability of  $\Psi$--Hilfer FDE \eqref{e11}. By utilizing generalized Gronwall inequality \cite{Sousa1} we obtain an estimations for the difference of two $\epsilon$-approximated solutions of $\Psi$--Hilfer fractional differential equations \eqref{e11}-\eqref{e12}, from which we can derive the results pertaining to uniqueness and dependence of solutions on the initial conditions. 

The paper is organized as follows.In section 2, some basic definitions and results concerning $\Psi$--Hilfer fractional derivative have been provided which are important for the development of the paper. Section 3 dealt with existence and uniqueness of solutions of the problem \eqref{e11} - \eqref{e12}. Section 4 dealt with HU stability of  $\Psi$--Hilfer FDE \eqref{e11} via successive approximations. In section 5, we study $\epsilon-$approximate solution of the $\Psi$--Hilfer FDE \eqref{e11}.In section 6, we provide an illustrative example.
% % % % % % % % % % % % % % % % % % % % % % % % % % % % % % % % % % % % % % % % % % % % % % % % %
\section{Preliminaries} \label{preliminaries}
In this section, we recall few  definitions, notions and the fundamental results about fractional integrals of a function with respect to another function \cite{Kilbas} and $\Psi$--Hilfer fractional operator \cite{Sousa1,Sousa}.

Let $0<a<b<\infty$, ~$\Delta=[a,b] \subset \mathbb{R}_{+}= [0,\infty), 0\leq \rho <1$ and $\Psi\in C^{1}(\Delta,\mathbb{R})$ be an increasing function such that $\Psi'(x)\neq 0$, $\forall~ x\in \Delta$. The weighted spaces $\mathbf{C}_{1-\rho;\,\Psi}(\Delta,\mathbb{R})$, ~$\mathbf{C}_{1-\rho;\,\Psi}^\rho (\Delta,\mathbb{R})$ and ~$\mathbf{C}_{1-\rho;\,\Psi}^{\mu,\nu}(\Delta,\mathbb{R})$ of functions are defined as follows:
\begin{itemize}
\item[(i)] $\mathbf{C}_{1-\rho;\, \Psi}(\Delta,\mathbb{R})=\left\{h:(a,b]\to\mathbb{R} : ~(\Psi(t)-\Psi(a))^{1-\rho}h(t)\in \mathbf{C}(\Delta,\mathbb{R})\right\}$, with the norm
$$||h||_{\mathbf{C}_{1-\rho;\, \Psi}}=\max_{t\in \Delta}\left|(\Psi(t)-\Psi(a))^{1-\rho}\,h(t)\right|,$$
\item[(ii)] $\mathbf{C}_{1-\rho;\, \Psi}^\rho (\Delta,\mathbb{R})=\left\{h\in \mathbf{C}_{1-\rho;\, \Psi}(\Delta,\mathbb{R}): \mathbf{D}_{a^+}^\rho h(t)\in \mathbf{C}_{1-\rho;\, \Psi}(\Delta,\mathbb{R}) \right\},$
\item[(iii)] $\mathbf{C}_{1-\rho;\, \Psi}^{\mu,\nu}(\Delta,\mathbb{R})=\left\{h\in \mathbf{C}_{1-\rho;\, \Psi}(\Delta,\mathbb{R}): ~ ^H\mathbf{D}_{a^+}^{\mu,\nu} h(t)\in \mathbf{C}_{1-\rho;\, \Psi}(\Delta,\mathbb{R}) \right\}.$
\end{itemize}
\begin{definition} [\cite{{Kilbas}},\cite{Samko}]
The $\Psi$--Riemann fractional integral of order $\mu > 0$ of the function h is given by 
\begin{equation*}\label{21}
\mathbf{I}_{a+}^{\mu ;\,\Psi }h\left( t\right) :=\frac{1}{\Gamma \left( \mu
\right) }\int_{a}^{t}\mathcal{L}_{\Psi}^{\mu}(t,\eta ) 
h\left( \eta\right) \, d\eta,
\end{equation*}
where
$$\mathcal{L}_{\Psi}^{\mu}(t,\eta )=\Psi ^{\prime }\left(\eta \right) \left( \Psi \left(
t\right) -\Psi \left( \eta \right) \right) ^{\mu-1}$$
\end{definition}
\begin{lemma}[\cite{Kilbas}]
Let $\mu>0$, $\nu>0$ and $\delta >0$. Then: 
\begin{itemize}
\item[(i)] $\mathbf{I}_{a^+}^{\mu;\, \Psi}\mathbf{I}_{a^+}^{\nu ;\, \Psi} h(t)=\mathbf{I}_{a^+}^{\mu+\nu;\, \Psi} h(t)$
\item[(ii)] If $h(t)= (\Psi(t)-\Psi(a))^{\delta-1},$
 then   ~ $\mathbf{I}_{a^+}^{\mu;\, \Psi}h(t)=\frac{\Gamma(\delta)}{\Gamma(\mu+\delta)}(\Psi(t)-\Psi(a))^{\mu + \delta-1}.$
\end{itemize}
\end{lemma}
We need following results  \cite{Kilbas,Samko} which are useful in the subsequent analysis of the paper.
\begin{lemma}[\cite{{Sousa}}] \label{lem22} 
If ${\mu}>0$ and $0\leq \rho <1$, then $\mathbf{I}_{a^+}^{\mu;\, \Psi}$ is bounded from  $\mathbf{C}_{\rho;\, \Psi}(\Delta,\mathbb{R})$ to $\mathbf{C}_{\rho;\, \Psi}(\Delta,\mathbb{R}).$ Also, if $\rho\leq \mu$, then $\mathbf{I}_{a^+}^{\mu;\, \Psi}$ is bounded from $\mathbf{C}_{\rho;\, \Psi}(\Delta,\mathbb{R})$ to $C(\Delta,\mathbb{R}).$
\end{lemma}
\begin{definition} [\cite{{Sousa1}}]
The $\Psi$--Hilfer fractional derivative of a function $h$ of order $0<\mu<1$ and type $0\leq \nu \leq 1$, is defined by
$$^H \mathbf{D}^{\mu, \, \nu; \, \Psi}_{a^+}h(t)= \mathbf{I}_{a^+}^{\nu ({1-\mu});\, \Psi} \left(\frac{1}{{\Psi}^{'}(t)}\frac{d}{dt}\right)^{'}\mathbf{I}_{a^+}^{(1-\nu)(1-\mu);\, \Psi} h(t).$$
\end{definition}
\begin{lemma}[\cite{Sousa1}] \label{lem2.3}
If $h\in C^{1}(\Delta,\mathbb{R}),$ $0<\mu<1$ and $0\leq\nu \leq 1 $, then
\begin{enumerate}
\item[(i)] $\mathbf{I}_{a^+}^{\mu;\, \Psi}\, {^H \mathbf{D}^{\mu, \, \nu; \, \Psi}_{a^+}}h(t)= h(t)- \Omega_{\Psi}^{\rho}(t,a)\mathbf{I}_{a^+}^{(1-\nu)(1-\mu);\Psi}h(a),$
where ~$\Omega_{\Psi}^{\rho}(t,a)=\frac{(\Psi(t)-\Psi(a))^{\rho-1}}{\Gamma(\rho)}
$
\item[(ii)] ${^H\mathbf{D}^{\mu, \, \nu; \, \Psi}_{a^+}}\,\mathbf{I}_{a^+}^{\mu;\, \Psi}h(t)=h(t).$
\end{enumerate}
 \end{lemma}
\begin{definition} [\cite{Kilbas}]
Let $\mu >0, \nu >0$.The one parameter Mittag-Leffler function is defined as
$$E_{\mu}(z)=\sum_{k=0}^{\infty}\frac{z^k}{\Gamma(k \mu +1)},$$
and the two parameter Mittag-Leffler function is defined as
$$E_{\mu,\nu}(z)=\sum_{k=0}^{\infty}\frac{z^k}{\Gamma(k \mu +\nu)}.$$
\end{definition}
\section{Existence and Uniqueness results}
In this section we derive the existence and uniqueness results of  the Cauchy-type problem \eqref{e11}-\eqref{e12} by utilizing the following  modified version of contraction principle.
\begin{lemma}[\cite{Siddiqi}] \label{lem31}
Let $\mathcal{X}$ be a Banach space and let $\mathcal{T}$ be an operator which maps the element of $\mathcal{X}$ into itself for which $\mathcal{T}^r$ is a contraction, where $r$ is a positive integer then $\mathcal{T}$ has a unique fixed point.
\end{lemma}
\begin{theorem} \label{th3.1}
 Let  $0<\mu<1$ and  $0\leq\nu\leq 1,$ and $\rho=\mu+\nu-{\mu \nu}.$
 Let $f:(a,b] \times \mathbb{R} \to \mathbb{R}$ be a function such that $f(\cdot,y(\cdot))\in \mathbf{C}_{{1-\rho};\,\Psi}(\Delta,\mathbb{R})$ for any $y\in \mathbf{C}_{{1-\rho};\, \Psi}(\Delta,\mathbb{R}),$ and let f satisfies the Lipschitz condition with respect to second argument
\begin{align}
 |f(t,y_1)-f(t,y_2)|\leq L |y_1-y_2|
\end{align}
for all $t\in (a,b]$ and for all $y_1,y_2\in \mathbb{R}$, where $L>0$ is a Lipschitz constant.Then the Cauchy problem \eqref{e11}-\eqref{e12} has unique solution in $\mathbf{C}_{{1-\rho};\, \Psi}(\Delta,\mathbb{R}).$  
\end{theorem} 
 \begin{proof}
The equivalent fractional integral to the  initial value problem \eqref{e11}-\eqref{e12} is given by \cite{Sousa1} 
\begin{align}
y(t)&=\Omega_{\Psi}^{\rho}(t,a)\,y_a+ \mathbf{I}_{a^+}^{\mu;\, \Psi}f(t,y(t))\nonumber\\
&=\Omega_{\Psi}^{\rho}(t,a)\,y_a+\frac{1}{\Gamma(\mu)}\int_{a}^{t}\mathcal{L}_{\Psi}^{\mu}(t,\eta) f(\eta, y(\eta))\, d\eta,\, t \in (a,b],\label{e13}
\end{align}
Our aim is to prove that the fractional integral \eqref{e13} has a solution in the weighted space $\mathbf{C}_{{1-\rho};\, \Psi}(\Delta,\mathbb{R})$. Consider the operator $\mathbb{T}$ defined on $:\mathbf{C}_{{1-\rho};\, \Psi}(\Delta,\mathbb{R})$ by
 \begin{align} \label{e14}
 (\mathbb{T}y)(t)&=\Omega_{\Psi}^{\rho}(t,a)\,y_a+\frac{1}{\Gamma(\mu)}\int_{a}^{t}\mathcal{L}_{\Psi}^{\mu}(t,\eta)  f(\eta, y(\eta))\, d\eta.
 \end{align}
 By lemma \ref{lem22}, it follows that $\mathbf{I}_{a^+}^{\mu;\, \Psi} f(.,y(.)) \in \mathbf{C}_{{1-\rho};\, \Psi}(\Delta,\mathbb{R})$. Clearly,\,${y_a} \,\Omega_{\Psi}^{\rho}(t,a) \in \mathbf{C}_{{1-\rho};\, \Psi}(\Delta,\mathbb{R})$.  Therefore, from \eqref{e14}, we have  $\mathbb{T}y \in \mathbf{C}_{{1-\rho};\, \Psi}(\Delta,\mathbb{R})$ for any $y\in \mathbf{C}_{{1-\rho};\, \Psi}(\Delta,\mathbb{R})$. This proves T maps $\mathbf{C}_{{1-\rho};\, \Psi}(\Delta,\mathbb{R})$ into itself.
Note that the fractional integral equation \eqref{e13} can be written as fixed point operator equation
$$
y=\mathbb{T}y, ~y\in \mathbf{C}_{{1-\rho};\, \Psi}(\Delta,\mathbb{R}).
$$ 
We prove that  the above operator equation has fixed point which will act as a solution for the problem \eqref{e11}-\eqref{e12}.  For any $t \in (a,b]$, consider the space $\mathbf{C}_{t;\, \Psi} = \mathbf{C}_{{1-\rho};\, \Psi}([a,t],\mathbb{R})$ with the norm defined by,  
 $$
 \|y\|_{\mathbf{C}_{t;\, \Psi}}=\max_{w\in[a,t]}\left |{(\Psi(w)-\Psi(a))}^{1-\rho} y(w)\right |.
 $$
 Using mathematical induction for any  $y_1,y_2\in \mathbf{C}_{t;\, \Psi}$ and $t \in (a,b]$, we prove that
\begin{align} \label{e3.5}
\| \mathbb{T}^jy_1-\mathbb{T}^jy_2\|_{\mathbf{C}_{t;\, \Psi}}\leq \Gamma(\rho)\frac{{(L{(\Psi(t)-\Psi(a))}^\mu)^j}}{\Gamma(j\mu+\rho)}\|y_1-y_2\|_{\mathbf{C}_{t;\, \Psi}},  ~j \in \mathbb{N}.
\end{align}
Let any $y_1,y_2\in \mathbf{C}_{t;\, \Psi}$. Then from the definition of operator $\mathbb{T}$ given in \eqref{e14} and using Lipschitz condition on $f$, we have
\begin{align*} 
&\| \mathbb{T}y_1-\mathbb{T}y_2\|_{\mathbf{C}_{t;\,\Psi}}\\
&= \max_{w\in[0,t]}\left|{(\Psi(w)-\Psi(a))}^{1-\rho}\left(\mathbb{T}y_1(w)-\mathbb{T}y_2(w)\right) \right|\\
&=\max_{w\in[0,t]}\left|{(\Psi(w)-\Psi(a))}^{1-\rho}\frac{1}{\Gamma(\mu)}\int_{a}^{w}\mathcal{L}_{\Psi}^{\mu}(w,\eta) \left( f(\eta, y_1(\eta))-f(\eta,y_2(\eta))\right) \, d\eta\right|\\
&\leq L \max_{w\in[0,t]}\left|{(\Psi(w)-\Psi(a))}^{1-\rho}\frac{1}{\Gamma(\mu)}\int_{a}^{w}\mathcal{L}_{\Psi}^{\mu}(w,\eta)\left|  y_1(\eta)-y_(\eta)\right| \, d\eta\right|\\
&= L \max_{w\in[0,t]}\left|{(\Psi(t)-\Psi(a))}^{1-\rho}\frac{1}{\Gamma(\mu)}\int_{a}^{w}\left\lbrace \mathcal{L}_{\Psi}^{\mu}(w,\eta){(\Psi(\eta)-\Psi(a))}^{\rho-1}\right\rbrace  \times \right.\\
&\left.\qquad\left\lbrace {(\Psi(\eta)-\Psi(a))}^{1-\rho}\left| y_1(\eta)-y_2(\eta)\right| \right\rbrace  \, d\eta\ \right|\\
&\leq \frac{L{(\Psi(t)-\Psi(a))}^{1-\rho}}{\Gamma(\mu)} \int_{a}^{t}\mathcal{L}_{\Psi}^{\mu}(t,\eta){(\Psi(\eta)-\Psi(a))}^{\rho-1}\times\\
&\qquad \max_{w\in[0,\eta]}\left| {(\Psi(w)-\Psi(a))}^{1-\rho}(y_1(w)-y_2(w))\right| \, d\eta\\
&\leq \frac{L{(\Psi(t)-\Psi(a))}^{1-\rho}}{\Gamma(\mu)} \|y_1-y_2\|_{c_{t;\, \Psi}} \int_{a}^{t}\mathcal{L}_{\Psi}^{\mu}(t,\eta){(\Psi(\eta)-\Psi(a))}^{\rho-1}\, d\eta \\
&\leq L \|y_1-y_2\|_{\mathbf{C}_{t;\, \Psi}}~ \left[{(\Psi(t)-\Psi(a))}^{1-\rho}\,\mathbf{I}_{a^+}^{\mu;\, \Psi} {(\Psi(t)-\Psi(a))}^{\rho-1}\right]\\
&=\frac{{(L{(\Psi(t)-\Psi(a))}^\mu)}\Gamma(\rho)}{\Gamma(\mu+\rho)}\|y_1-y_2\|_{\mathbf{C}_{t;\, \Psi}}
\end{align*}
Thus the inequality \eqref{e3.5} holds for $j=1$. Let us suppose that the inequality \eqref{e3.5} holds for $j=r\in\N,$ i.e. suppose
\begin{align} \label{e3.6}
\| \mathbb{T}^ry_1-\mathbb{T}^ry_2\|_{\mathbf{C}_{t;\, \Psi}}\leq \Gamma(\rho) \frac{{(L{(\Psi(t)-\Psi(a))}^\mu)^r}}{\Gamma(r\mu+\rho)}\|y_1-y_2\|_{\mathbf{C}_{t;\, \Psi}}
\end{align}
holds.  Next, we prove that \eqref{e3.5} holds for $j=r+1.$ Let $y_1,y_2\in \mathbf{C}_{t;\, \Psi}$ and denote $y_1^*=\mathbb{T}^ry_1$ and $y_2^*=\mathbb{T}^ry_2.$ Then using the definition of operator T and Lipschitz condition on $f$, we get
\begin{align*}
&\| \mathbb{T}^{r+1}y_1-\mathbb{T}^{r+1}y_2\|_{\mathbf{C}_{t;\, \Psi}}\\
&=\| \mathbb{T}(\mathbb{T}^{r}y_1)-\mathbb{T}(\mathbb{T}^{r}y_2)\|_{\mathbf{C}_{t;\, \Psi}}\\
&=  \| \mathbb{T}y_1^*-\mathbb{T}y_2^*\|_{\mathbf{C}_{t;\, \Psi}} \\
&= \max_{w\in[a,t]}\left| {(\Psi(w)-\Psi(a))}^{1-\rho}\left( \mathbb{T}y_1^*(w)-\mathbb{T}y_2^*(w)\right) \right| \\
&=\max_{w\in[a,t]}\left|{(\Psi(w)-\Psi(a))}^{1-\rho}\frac{1}{\Gamma(\mu)}\int_{a}^{w}\mathcal{L}_{\Psi}^{\mu}(w,\eta) \left( f(\eta, y_1^*(\eta))-f(\eta,y_2^*(\eta)\right)\, d\eta\right|\\
&\leq L \max_{w\in[a,t]}\left\lbrace {(\Psi(w)-\Psi(a))}^{1-\rho}\frac{1}{\Gamma(\mu)}\int_{a}^{w}\left( \mathcal{L}_{\Psi}^{\mu}(w,\eta){(\Psi(\eta)-\Psi(a))}^{\rho-1}\right)  \times \right.\\
& \left. \qquad \left( {(\Psi(\eta)-\Psi(a))}^{1-\rho}\left| y_1^*(\eta)-y_2^*(\eta)\right| \right)\, d\eta\right\rbrace \\
&\leq \frac{L{(\Psi(t)-\Psi(a))}^{1-\rho}}{\Gamma(\mu)}  \int_{a}^{t}\left(\left\{\mathcal{L}_{\Psi}^{\mu}(t,\eta){(\Psi(\eta)-\Psi(a))}^{\rho-1}\right\} \times \right.\\
&\left.\qquad  \max_{w\in[a,\eta]}\left|{(\Psi(w)-\Psi(a))}^{1-\rho}\left(  y_1^*(w)-y_2^*(w)\right) \right|\right)\, d\eta\\
&\leq \frac{L{(\Psi(t)-\Psi(a))}^{1-\rho}}{\Gamma(\mu)}\int_{a}^{t} \mathcal{L}_{\Psi}^{\mu}(t,\eta){(\Psi(\eta)-\Psi(a))}^{\rho-1}\|y_1^*-y_2^*\|_{\mathbf{C}_{\eta;\, \Psi}}\, d\eta
\end{align*}
From (\ref{e3.6}), we have
\begin{align*} 
\|y_1^*-y_2^*\|_{\mathbf{C}_{s;\, \Psi}}=\| \mathbb{T}^ry_1-\mathbb{T}^ry_2\|_{\mathbf{C}_{s;\, \Psi}}\leq \Gamma(\rho)\frac{{(L{(\Psi(s)-\Psi(a))}^\mu)^r}}{\Gamma(r\mu+\rho)}\|y_1-y_2\|_{\mathbf{C}_{s;\, \Psi}}
\end{align*}
Therefore,
\begin{align*}
&\| \mathbb{T}^{r+1}y_1-\mathbb{T}^{r+1}y_2\|_{\mathbf{C}_{t;\, \Psi}}\\
&\leq \frac{L{(\Psi(t)-\Psi(a))}^{1-\rho}}{\Gamma(\mu)}\int_{a}^{t}\mathcal{L}_{\Psi}^{\mu}(t,\eta){(\Psi(\eta)-\Psi(a))}^{\rho-1} \times\\
&\qquad\Gamma(\rho)\frac{{(L{(\Psi(\eta)-\Psi(a))}^\mu)^r}}{\Gamma(r\mu+\rho)}\|y_1-y_2\|_{\mathbf{C}_{\eta;\, \Psi}}\, d\eta\\
&\leq \left( \frac{L^{r+1}\Gamma(\rho)}{{\Gamma(r\mu+\rho)}} \|y_1-y_2\|_{\mathbf{C}_{t;\, \Psi}}\right)  \times\\
& \qquad\left({(\Psi(t)-\Psi(a))}^{1-\rho}\frac{1}{{\Gamma(\mu)}}\int_{a}^{t}\mathcal{L}_{\Psi}^{\mu}(t,\eta){(\Psi(\eta)-\Psi(a))}^{r\mu+\rho-1}\, d\eta\right) \\
&\leq \frac{L^{r+1}\Gamma(\rho)}{\Gamma(r\mu+\rho)}\|y_1-y_2\|_{\mathbf{C}_{t;\, \Psi}}{(\Psi(t)-\Psi(a))}^{1-\rho}\mathbf{I}_{a^+}^\mu {(\Psi(t)-\Psi(a))}^{r\mu+\rho-1}\\
&= \frac{L^{r+1}\Gamma(\rho)}{\Gamma(r\mu+\rho)}\|y_1-y_2\|_{\mathbf{C}_{t;\, \Psi}}{(\Psi(t)-\Psi(a))}^{1-\rho} \frac{\Gamma(r\mu+\rho)}{\Gamma((r+1)\mu+\rho)}\, {(\Psi(t)-\Psi(a))}^{(r+1)\mu+\rho-1}\\
&=\Gamma(\rho)\frac{(L{(\Psi(t)-\Psi(a))}^\mu)^{r+1}}{\Gamma((r+1)\mu+\rho)}\|y_1-y_2\|_{\mathbf{C}_{t;\, \Psi}}
\end{align*}
Thus we have
$$
\| \mathbb{T}^{r+1}y_1-\mathbb{T}^{r+1}y_2\|_{\mathbf{C}_{t;\, \Psi}} \leq \Gamma(\rho)\frac{(L{(\Psi(t)-\Psi(a))}^\mu)^{r+1}}{\Gamma((r+1)\mu+\rho)}\|y_1-y_2\|_{\mathbf{C}_{t;\, \Psi}}
$$
Therefore, by principle of mathematical induction the inequality \eqref{e3.5} holds for all $j\in \N$ and for every t in $\Delta$. As a consequence we find on the fundamental interval $\Delta$,
\begin{align} 
\| \mathbb{T}^jy_1-\mathbb{T}^jy_2\|_{\mathbf{C}_{{1-\rho};\, \Psi}(\Delta,\mathbb{R})}\leq \Gamma(\rho) \frac{{(L{(\Psi(b)-\Psi(a))}^\mu)^j}}{\Gamma(j\mu+\rho)}\|y_1-y_2\|_{\mathbf{C}_{{1-\rho};\, \Psi}(\Delta,\mathbb{R})}
\end{align}
By definition of two parameter Mittag-Leffler function, we have
$$E_{\mu,\rho}(L{(\Psi(b)-\Psi(a))}^\mu)=\sum_{j=0}^{\infty}\frac{(L{(\Psi(b)-\Psi(a))}^\mu)^j}{\Gamma(j\mu+\rho)}$$\\
Note that $\frac{(L{(\Psi(b)-\Psi(a))}^\mu)^j}{\Gamma(j\mu+\rho)}$ is the $j^{th}$ term of the convergent series of real numbers. Therefore,
$$\lim_{j\to\infty}\frac{(L{(\Psi(b)-\Psi(a))}^\mu)^j}{\Gamma(j\mu+\rho)}=0.$$ Thus we can choose $j\in\N$ such that 
$${\Gamma(\rho)}\frac{(L{(\Psi(b)-\Psi(a))}^\mu)^j}{\Gamma(j\mu+\rho)}<1,$$\\
so that $\mathbb{T}^j$ is a contraction.  Therefore, by Lemma \ref{lem31}, $\mathbb{T}$ has a unique fixed point $y^* $ in $\mathbf{C}_{{1-\rho};\, \Psi}(\Delta,\mathbb{R}),$ which is a unique solution of the Cauchy type problem \eqref{e11}-\eqref{e12}.
\end{proof} 
 \begin{rem}
The existence result proved above with no restriction on the interval $\Delta=[a,b]$, and hence solution $y^*$ of \eqref{e11}-\eqref{e12} exists for any $a,b ~(0<a<b<\infty)$. Thus the Theorem \ref{th3.1} guarantees global unique solution in ${\mathbf{C}_{{1-\rho};\, \Psi}}(\Delta,\mathbb{R}) $.
 \end{rem} 
\section{Ulam-Hyers stability}
To discuss HU and HUR stability of \eqref{e11}, we adopt the approach of \cite{Wang,Rus}. For $\epsilon>0$ and continuous function $\phi:\Delta \to[0,\infty)$, we consider the following inequalities:
\begin{align}
\left|^H \mathbf{D}^{\mu,\nu;\,\Psi}_{a^+}y^{*}(t)-f(t,y^{*}(t))\right|&\leq \epsilon,~ t \in \Delta \label{ine1}\\
\left|^H \mathbf{D}^{\mu,\nu;\,\Psi}_{a^+}y^{*}(t)-f(t,y^{*}(t))\right|&\leq \phi(t),~t \in \Delta \label{ine2}\\
\left|^H \mathbf{D}^{\mu,\nu;\,\Psi}_{a^+}y^{*}(t)-f(t,y^{*}(t))\right|&\leq \epsilon \phi(t)\label{ine3},~ t \in \Delta
\end{align}
\begin{definition}
The problem \eqref{e11} has HU stability if there exists a real number $\mathbf{C}_{f}>0$ such that for each $\epsilon>0$ and  for each solution
$y^{*}\in \mathbf{C}_{{1-\rho};\, \Psi}(\Delta,\mathbb{R})$ of the inequation \eqref{ine1} there exists a solution  $y \in \mathbf{C}_{{1-\rho};\, \Psi}(\Delta,\mathbb{R})$ of \eqref{e11} with
$$\|y^{*}-y\|_{\mathbf{C}_{{1-\rho};\, \Psi}(\Delta,\mathbb{R})}\leq C_{f} \,\epsilon. $$
\end{definition}
\begin{definition}
The problem \eqref{e11} has generalized HU stability if there exists a function $C_{f}\in ([0,\infty)),[0,\infty))$ with $C_f(0)=0$ such that for each solution
$y^{*}\in \mathbf{C}_{{1-\rho};\, \Psi}(\Delta,\mathbb{R})$ of the inequation \eqref{ine1} there exists a solution  $y\in \mathbf{C}_{{1-\rho};\, \Psi}(\Delta,\mathbb{R})$ of \eqref{e11} with
$$\|y^{*}-y\|_{\mathbf{C}_{{1-\rho};\, \Psi}(\Delta,\mathbb{R})}\leq C_{f}(\epsilon). $$
\end{definition}
\begin{definition}
The problem \eqref{e11} has HUR stability with respect to a function $\phi$ if there exists a real number $C_{f,\phi}>0$ such that for each solution
$y^{*}\in \mathbf{C}_{{1-\rho};\, \Psi}(\Delta,\mathbb{R})$ of the inequation \eqref{ine3} there exists a solution  $y\in \mathbf{C}_{{1-\rho};\, \Psi}(\Delta,\mathbb{R})$ of \eqref{e11} with
$$|(\Psi(t)-\Psi(a))^{1-\rho}(y^{*}(t)-y(t))| \leq C_{f,\phi}\,\epsilon \,\phi(t), ~t \in (\Delta,\mathbb{R}). $$
\end{definition}
\begin{definition}
The problem \eqref{e11} has generalized HUR stability with respect to a function $\phi$ if there exists a real number $C_{f,\phi}>0$ such that for each solution
$y^{*}\in \mathbf{C}_{{1-\rho};\, \Psi}(\Delta,\mathbb{R})$ of the inequation \eqref{ine2} there exists a solution  $y\in \mathbf{C}_{{1-\rho};\, \Psi}(\Delta,\mathbb{R})$ of \eqref{e11} with
$$|(\Psi(t)-\Psi(a))^{1-\rho}(y^{*}(t)-y(t))| \leq C_{f,\phi} \phi(t), ~~~t \in \Delta. $$
\end{definition}
In the next theorem we will make use of the successive approximation method to prove that the $\Psi$--Hilfer FDE \eqref{e11} is HU stable.
\begin{theorem}\label{thm4.2}
 Let $f:(a,b] \times \mathbb{R} \to \mathbb{R} $ be a function such that $f(\cdot,y(\cdot))\in \mathbf{C}_{{1-\rho};\, \Psi}(\Delta,\mathbb{R})$ for any $y\in \mathbf{C}_{{1-\rho};\, \Psi}(\Delta,\mathbb{R})$ and that satisfies the Lipschitz condition 
\begin{align*}
 |f(t,y_1)-f(t,y_2)|\leq L |y_1-y_2|, ~t\in (a,b], ~y_1,y_2\in \mathbb{R},
\end{align*}
where $L>0$ is a constant.
  For every $\epsilon>0,$ if $y^{*}\in \mathbf{C}_{{1-\rho};\, \Psi}(\Delta,\mathbb{R})$ satisfies 
  $$\left|^H \mathbf{D}^{\mu,\nu;\,\Psi}_{a^+}y^{*}(t)-f(t,y^{*}(t))\right|\leq \epsilon,~ t \in \Delta$$
  then there exists a solution $y$ of equation \eqref{e11} in $\mathbf{C}_{{1-\rho};\, \Psi}(\Delta,\mathbb{R})$ with $\mathbf{I}_{a^+}^{{1-\rho};\, \Psi}y^{*}(a)=\mathbf{I}_{a^+}^{{1-\rho};\, \Psi}y(a)$ such that 
  $$\|y^{*}-y\|_{\mathbf{C}_{{1-\rho};\, \Psi}(\Delta,\mathbb{R})}\leq \left(\frac{(E_\mu(L(\Psi(b)-\Psi(a))^\mu)-1)}{L}(\Psi(b)-\Psi(a))^{1-\rho}\right) \epsilon,  ~t\in \Delta.$$ 
 \end{theorem} 
 \begin{proof}
Fix any $\epsilon>0,$ let $z\in \mathbf{C}_{{1-\rho};\, \Psi}(\Delta,\mathbb{R})$ satisfies 
\begin{align}\label{e4.40}
\left|^H \mathbf{D}^{\mu,\nu;\,\Psi}_{a^+}y^{*}(t)-f(t,y^{*}(t))\right|\leq \epsilon,~~ t \in \Delta.
\end{align}
Then there exists a function $\sigma_{y^{*}} \in \mathbf{C}_{{1-\rho};\, \Psi}(\Delta,\mathbb{R})$ ( depending on $y^{*}$ ) such that $|\sigma_{y^{*}}(t)|\leq\epsilon,~t\in \Delta$ and
\begin{align}\label{e4.4}
^H \mathbf{D}^{\mu,\nu;\,\Psi}_{a^+}y^{*}(t) = f(t,{y^{*}}(t))+ \sigma_{y^{*}}(t), ~t\in \Delta.
\end{align}
If $y^{*}(t)$ satisfies \eqref{e4.4} then it satisfies equivalent fractional integral equation
\begin{align}\label{e4.5}
y^{*}(t)& =\Omega_{\Psi}^{\rho}(t,a)\,
\mathbf{I}_{a^+}^{{1-\rho};\, \Psi}y^{*}(a) +\frac{1}{\Gamma(\mu)}\int_{a}^{t}\mathcal{L}_{\Psi}^{\mu}(t,\eta)  f(\eta, y^{*}(\eta))\, d\eta \nonumber\\
&+ \frac{1}{\Gamma(\mu)}\int_{a}^{t}\mathcal{L}_{\Psi}^{\mu}(t,\eta) \sigma_{y^{*}}(\eta)\, d\eta, ~t \in \Delta.
\end{align}
Define
\begin{align}\label{e4.6}
y_0(t)=y^{*}(t),   ~~t\in \Delta 
\end{align}
and consider the sequence ${\{y_n\}}_{n=1}^{\infty}\subseteq \mathbf{C}_{{1-\rho};\, \Psi}(\Delta,\mathbb{R})$ defined by
\begin{align}\label{e4.7}
y_n(t)= \Omega_{\Psi}^{\rho}(t,a)\, \mathbf{I}_{a^+}^{{1-\rho};\, \Psi}y^{*}(a)+\frac{1}{\Gamma(\mu)}\int_{a}^{t}\mathcal{L}_{\Psi}^{\mu}(t,\eta)  f(\eta, y_{n-1}(\eta))\, d\eta, ~t \in \Delta .
\end{align}
Using mathematical induction firstly we prove that for every $t\in \Delta$ and 
$y_j\in \mathbf{C}_{{1-\rho};\, \Psi}[a,t]=\mathbf{C}_{t;\, \Psi}$
\begin{align}\label{ineq4.8}
\| y_j-y_{j-1}\|_{\mathbf{C}_{t;\, \Psi}}\leq \frac{\epsilon}{L}\frac{(L {(\Psi(t)-\Psi(a))}^\mu)^j}{\Gamma(j\mu+1)}\,{(\Psi(t)-\Psi(a))}^{1-\rho}, ~j\in\N.
\end{align}
By definition of successive approximations and using \eqref{e4.5} we have
\begin{align*} 
&\| y_1-y_0\|_{\mathbf{C}_{t;\, \Psi}} \\
&= \max_{w\in[a,t]}|{(\Psi(w)-\Psi(a))}^{1-\rho}\left\lbrace  y_1(w)-y_0(w)\right\rbrace |\\
&=\max_{w\in[0,t]}\left|{(\Psi(w)-\Psi(a))}^{1-\rho}\left(\Omega_{\Psi}^{\rho}(w,a)\,\mathbf{I}_{a^+}^{{1-\rho};\, \Psi}z(a)+\mathbf{I}_{a+}^{\mu;\, \Psi}f(w, y_0(w))-y_0(w)\right)\right|\\
&=\max_{w\in[0,t]}\left|{(\Psi(w)-\Psi(a))}^{1-\rho}\left(\Omega_{\Psi}^{\rho}(w,a)\,\mathbf{I}_{a^+}^{{1-\rho};\, \Psi}z(a)+\mathbf{I}_{a+}^{\mu;\, \Psi}f(w, z(w))-z(w)\right)\right|\\
&=\max_{w\in[0,t]}\left|{(\Psi(w)-\Psi(a))}^{1-\rho}\frac{1}{\Gamma(\mu)}\int_{a}^{w}\mathcal{L}_{\Psi}^{\mu}(w,\eta) \sigma_z(\eta)\, d\eta\right|\\
&\leq\max_{w\in[0,t]}\left[{(\Psi(w)-\Psi(a))}^{1-\rho}\frac{1}{\Gamma(\mu)}\int_{a}^{w}\mathcal{L}_{\Psi}^{\mu}(w,\eta) |\sigma_z(\eta)|\, d\eta\right]\\
&\leq\epsilon\max_{w\in[0,t]}\left[{(\Psi(w)-\Psi(a))}^{1-\rho}\frac{1}{\Gamma(\mu)}\int_{a}^{w}\mathcal{L}_{\Psi}^{\mu}(w,\eta)\, d\eta\right]\\
&\leq\frac{\epsilon}{\Gamma(\mu+1)} {(\Psi(t)-\Psi(a))}^{1-\rho} {(\Psi(t)-\Psi(a))}^\mu\\
&=\frac{\epsilon}{L}\frac{(L{(\Psi(t)-\Psi(a))}^\mu)}{\Gamma(\mu+1)}{(\Psi(t)-\Psi(a))}^{1-\rho}
\end{align*}
Therefore, $$\| y_1-y_0\|_{\mathbf{C}_{t;\, \Psi}}\leq \frac{\epsilon}{L}\frac{(L{(\Psi(t)-\Psi(a))}^\mu)}{\Gamma(\mu+1)}{(\Psi(t)-\Psi(a))}^{1-\rho}$$
which proves the inequality \eqref{ineq4.8} for j = 1.
Let us suppose that the inequality \eqref{ineq4.8} holds for $j=r\in\N,$ we prove it for $j=r+1$.  By definition of successive approximations and Lipschitz condition on $f$, we obtain
\begin{align*} 
&\| y_{r+1}-y_r\|_{\mathbf{C}_{t;\, \Psi}}\\
&= \max_{w\in[0,t]}\left|{(\Psi(w)-\Psi(a))}^{1-\rho}\left\{y_{r+1}(w)-y_r(w)\right\}\right|\\
&=\max_{w\in[0,t]}\left|{(\Psi(w)-\Psi(a))}^{1-\rho}\left(\mathbf{I}_{a+}^{\mu;\, \Psi}f(w, y_r(w))-\mathbf{I}_{a+}^{\mu;\, \Psi}f(w, y_{r-1}(w))\right)\right|\\
&\leq L\max_{w\in[0,t]}\left[{(\Psi(w)-\Psi(a))}^{1-\rho}\frac{1}{\Gamma(\mu)}\int_{a}^{w}\mathcal{L}_{\Psi}^{\mu}(w,\eta)\,\left|y_r(\eta)-y_{r-1}(\eta)\right|\, d\eta\right]\\
&\leq \frac{L{(\Psi(t)-\Psi(a))}^{1-\rho}}{\Gamma(\mu)} \int_{a}^{t}\mathcal{L}_{\Psi}^{\mu}(t,\eta){(\Psi(\eta)-\Psi(a))}^{\rho-1} \times\\
&~~~~~~\max_{w\in[0,\eta]}\left|{(\Psi(w)-\Psi(a))}^{1-\rho}\left\lbrace y_r(w)-y_{r-1}(w)\right\rbrace \right|\, d\eta \\
&= \frac{L{(\Psi(t)-\Psi(a))}^{1-\rho}}{\Gamma(\mu)}  \int_{a}^{t}\mathcal{L}_{\Psi}^{\mu}(t,\eta){(\Psi(\eta)-\Psi(a))}^{\rho-1}
\|y_r-y_{r-1}\|_{\mathbf{C}_{\eta;\, \Psi}}\, d\eta
\end{align*}
Using the inequality \eqref{ineq4.8} for j=r, we have
\begin{align*} 
\| y_{r+1}-y_r\|_{\mathbf{C}_{t;\, \Psi}}
&\leq \frac{L{(\Psi(t)-\Psi(a))}^{1-\rho}}{\Gamma(\mu)}  \int_{a}^{t}\mathcal{L}_{\Psi}^{\mu}(t,\eta){(\Psi(\eta)-\Psi(a))}^{\rho-1} \times\\
&\qquad \left( \frac{\epsilon}{L}\frac{(L{(\Psi(\eta)-\Psi(a))}^\mu)^r}{\Gamma(r\mu+1)}{(\Psi(\eta)-\Psi(a))}^{1-\rho}\right) \, d\eta\\
&=\frac{\epsilon}{L}\frac{L^{r+1}}{\Gamma(r\mu+1)}{(\Psi(t)-\Psi(a))}^{1-\rho}\mathbf{I}_{a^+}^{\mu;\, \Psi} {(\Psi(t)-\Psi(a))}^{r\mu}\\
&=\frac{\epsilon}{L}\frac{L^{r+1}}{\Gamma(r\mu+1)}{(\Psi(t)-\Psi(a))}^{1-\rho}\frac{\Gamma(r\mu+1)}{\Gamma((r+1)\mu+1)}\,{(\Psi(t)-\Psi(a))}^{(r+1)\mu}
\end{align*}
Therefore,
$$\| y_{r+1}-y_r\|_{\mathbf{C}_{t;\, \Psi}}\leq \frac{\epsilon}{L}\frac{(L{(\Psi(t)-\Psi(a))}^\mu)^{r+1}}{\Gamma((r+1)\mu+1)}{(\Psi(t)-\Psi(a))}^{1-\rho},$$
which is the inequality (\ref{ineq4.8}) for $j=r+1.$ Using principle of mathematical induction the inequality (\ref{ineq4.8}) holds for every $j\in \N$ and every $t\in \Delta.$\\
Therefore,~~~~$$\| y_j-y_{j-1}\|_{\mathbf{C}_{{1-\rho};\, \Psi}(\Delta,\mathbb{R})}\leq \frac{\epsilon}{L}\frac{(L{(\Psi(b)-\Psi(a))}^\mu)^{j}}{\Gamma(j\mu+1)}{(\Psi(b)-\Psi(a))}^{1-\rho}$$ \\ Now using this estimation we have
\begin{align*} 
\sum_{j=1}^{\infty}\| y_j-y_{j-1}\|_{\mathbf{C}_{{1-\rho};\, \Psi}(\Delta,\mathbb{R})}\leq \frac{\epsilon}{L}{(\Psi(b)-\Psi(a))}^{1-\rho}\sum_{j=1}^{\infty}\frac{(L{(\Psi(b)-\Psi(a))}^\mu)^j}{\Gamma(j\mu+1)}
\end{align*}
Thus we have
\begin{align} 
\sum_{j=1}^{\infty}\| y_j-y_{j-1}\|_{\mathbf{C}_{{1-\rho};\, \Psi}(\Delta,\mathbb{R})}\leq \frac{\epsilon}{L}{(\Psi(b)-\Psi(a))}^{1-\rho}\left( E_\mu(L{(\Psi(b)-\Psi(a))}^\mu)-1\right) 
\end{align}
Hence the series
\begin{align}\label{s1} 
y_0+\sum_{j=1}^{\infty}(y_j-y_{j-1})
\end{align}
converges in the weighted space ${\mathbf{C}_{{1-\rho};\, \Psi}}(\Delta,\mathbb{R})$.
Let $y\in {\mathbf{C}_{{1-\rho};\, \Psi}}(\Delta,\mathbb{R})$ such that
\begin{align} \label{e4.11}
y=y_0+\sum_{j=1}^{\infty}(y_j-y_{j-1}),
\end{align}
Noting that $$y_n=y_0+\sum_{j=1}^{n}(y_j-y_{j-1})$$ is the $n^{th}$ partial sum of the series \eqref{s1}, we have 
$$
\| y_n-y\|_{\mathbf{C}_{{1-\rho};\, \Psi}(\Delta,\mathbb{R})}\to 0   ~as ~n\to \infty.
$$ 
Next, we prove that this limit function  $y$  is the solution of fractional integral equation with $\mathbf{I}_{a^+}^{{1-\rho};\, \Psi}y^{*}(a)=\mathbf{I}_{a^+}^{{1-\rho};\, \Psi}y(a)$. 
Next, by the definition of successive approximation, for any $t\in \Delta$, we have
\begin{small} 
\begin{align*} 
&\left|(\Psi(t)-\Psi(a))^{1-\rho}\left(y(t)-\Omega_{\Psi}^{\rho}(t,a)\,\mathbf{I}_{a^+}^{{1-\rho};\, \Psi}y(a)-\frac{1}{\Gamma(\mu)}\int_{a}^{t}\mathcal{L}_{\Psi}^{\mu}(t,\eta) f(\eta, y(\eta))\, d\eta\right)\right|\\
&=\left|(\Psi(t)-\Psi(a))^{1-\rho}\left(y(t)-\Omega_{\Psi}^{\rho}(t,a)\,\mathbf{I}_{a^+}^{{1-\rho};\, \Psi}z(a)-\frac{1}{\Gamma(\mu)}\int_{a}^{t}\mathcal{L}_{\Psi}^{\mu}(t,\eta)\,f(\eta,y(\eta))\, d\eta\right)\right|\\
&=\left|(\Psi(t)-\Psi(a))^{1-\rho}\left(y(t)-y_n(t)+\frac{1}{\Gamma(\mu)}\int_{a}^{t}\mathcal{L}_{\Psi}^{\mu}(t,\eta)\,  f(\eta, y_{n-1}(\eta))\, d\eta\right.\right.\\
&\left.\left. \qquad -\frac{1}{\Gamma(\mu)}\int_{a}^{t}\mathcal{L}_{\Psi}^{\mu}(t,\eta)\,  f(\eta,y(\eta))\, d\eta\right)\right|\\
&\leq \left|(\Psi(t)-\Psi(a))^{1-\rho}\left\lbrace y(t)-y_n(t)\right\rbrace \right|+ \left|(\Psi(t)-\Psi(a))^{1-\rho}\,\mathbf{I}_{a^+}^{\mu;\, \Psi}\{f(t,y_{n-1}(t))-f(t,y(t))\}\right|\\
&\leq \left\| y-y_n\right\|_{\mathbf{C}_{{1-\rho};\, \Psi}[a,b]}+L\left[(\Psi(t)-\Psi(a))^{1-\rho}\frac{1}{\Gamma(\mu)}\int_{a}^{t}\mathcal{L}_{\Psi}^{\mu}(t,\eta)\, |y_{n-1}(\eta)-y(\eta)|\, d\eta\right]\\
&\leq \| y-y_n\|_{\mathbf{C}_{{1-\rho};\, \Psi}[a,b]}+ L\|y_{n-1}-y\|_{\mathbf{C}_{{1-\rho};\, \Psi}[a,b]}\, (\Psi(t)-\Psi(a))^{1-\rho}\,\mathbf{I}_{a^+}^{\mu;\, \Psi}(\Psi(t)-\Psi(a))^{\rho-1} \\
&=\|y-y_n\|_{\mathbf{C}_{{1-\rho};\, \Psi}[a,b]}+\left(\frac{L\Gamma\rho}{\Gamma(\mu+\rho)}\,(\Psi(t)-\Psi(a))^\mu\right) \|y_{n-1}-y\|_{\mathbf{C}_{{1-\rho};\, \Psi}[a,b]},~ \forall n \in \mathbb{N}
\end{align*}
\end{small} 
By taking limit as $n \to \infty$ in the above inequality, for all $t\in [a,b]$, we obtain
\begin{small}
\begin{equation*} 
\left|(\Psi(t)-\Psi(a))^{1-\rho}\left(y(t)-\Omega_{\Psi}^{\rho}(t,a)\,\mathbf{I}_{a^+}^{{1-\rho};\, \Psi}y(a)-\frac{1}{\Gamma(\mu)}\int_{a}^{t}\mathcal{L}_{\Psi}^{\mu}(t,\eta)\,  f(\eta, y(\eta))\, d\eta\right)\right|= 0.
\end{equation*}
\end{small}
Since, $(\Psi(t)-\Psi(a))^{1-\rho}\neq 0$ for all $t\in \Delta$, we have  
\begin{align} \label{ineq4.12}
 y(t)=\Omega_{\Psi}^{\rho}(t,a) \,\mathbf{I}_{a^+}^{{1-\rho};\, \Psi}y(a)+\frac{1}{\Gamma(\mu)}\int_{a}^{t}\mathcal{L}_{\Psi}^{\mu}(t,\eta)\, f(\eta, y(\eta))\, d\eta,~ t\in \Delta.
\end{align}
This proves that $y$ is the solution of \eqref{e11}-\eqref{e12} in ${\mathbf{C}_{{1-\rho};\, \Psi}(\Delta,\mathbb{R})}$.
Further, for the solution $y^{*}$ of inequation \eqref{e4.40} and the solution $y$ of the equation \eqref{e11}, using  \eqref{e4.6} and  \eqref{e4.11}, for any $t\in \Delta,$ we have
\begin{align*} 
|(\Psi(t)-\Psi(a))^{1-\rho}(y^{*}(t)-y(t))|
&=\left|(\Psi(t)-\Psi(a))^{1-\rho}\left[y_0(t)-\left(y_0(t)+\sum_{j=1}^{\infty}( y_j(t)-y_{j-1}(t))\right)\right]\right|\\
&\leq \sum_{j=1}^{\infty}\left|(\Psi(t)-\Psi(a))^{1-\rho}(y_j(t)-y_{j-1}(t))\right|\\
&\leq \sum_{j=1}^{\infty}\| y_j-y_{j-1}\|_{\mathbf{C}_{1-\rho}[a,b]}\\
&\leq \frac{\epsilon}{L}(\Psi(b)-\Psi(a))^{1-\rho}(E_\mu(L(\Psi(b)-\Psi(a))^\mu)-1)
\end{align*}
Therefore, $$\|y^{*}-y\|_{\mathbf{C}_{{1-\rho};\, \Psi}[a,b]}\leq \left(\frac{(E_\mu(L(\Psi(b)-\Psi(a))^\mu)-1)}{L}(\Psi(b)-\Psi(a))^{1-\rho}\right) \epsilon $$
This proves the equation \eqref{e11} is HU stable.
\end{proof}
\begin{cor}\label{cor4.2}
Suppose that the function $f$ satisfies the assumptions of Theorem \ref{thm4.2}. 
Then, the problem \eqref{e11} is generalized HU stable. 
\end{cor} 
\begin{proof}
Set $$
\Psi_f(\epsilon)=\left(\frac{(E_\mu(L(\Psi(b)-\Psi(a))^\mu)-1)}{L}(\Psi(b)-\Psi(a))^{1-\rho}\right)\epsilon,
$$ 
in the proof of Theorem \ref{thm4.2}. Then $\Psi_f(0)=0$ and for each  $y^{*}\in \mathbf{C}_{{1-\rho};\, \Psi}(\Delta,\mathbb{R})$ that satisfies 
  the inequality $$\left|^H \mathbf{D}^{\mu,\nu;\,\Psi}_{a^+}y^{*}(t)-f(t,y^{*}(t))\right|\leq \epsilon,~ t \in \Delta,$$
 there exists a solution $y$ of equation \eqref{e11} in $\mathbf{C}_{{1-\rho};\, \Psi}(\Delta,\mathbb{R})$ with $\mathbf{I}_{a^+}^{{1-\rho};\, \Psi}y^{*}(a)=\mathbf{I}_{a^+}^{{1-\rho};\, \Psi}y(a)$ such that 
  $$\|y^{*}-y\|_{\mathbf{C}_{{1-\rho};\, \Psi}(\Delta,\mathbb{R})}\leq \Psi_f(\epsilon),~t\in \Delta.$$ 
Hence fractional differential equation \eqref{e11} is generalized HU stable.
\end{proof}
\begin{theorem}\label{Thm4.3}
Let $f:(a,b] \times \mathbb{R}\to \mathbb{R} $ be a function such that $f(\cdot,y(\cdot))\in \mathbf{C}_{{1-\rho};\, \Psi}(\Delta,\mathbb{R})$ for any $y\in \mathbf{C}_{{1-\rho};\, \Psi}(\Delta,\mathbb{R})$ and that satisfies the Lipschitz condition 
\begin{align*}
 |f(t,y_1)-f(t,y_2)|\leq L |y_1-y_2|, ~t\in (a,b], ~y_1,y_2\in \mathbb{R},
\end{align*}
where $L>0$ is a constant.
  For every $\epsilon>0,$ if $y^{*}\in \mathbf{C}_{{1-\rho};\, \Psi}(\Delta,\mathbb{R})$ satisfies 
  $$\left|^H \mathbf{D}^{\mu,\nu;\,\Psi}_{a^+}y^{*}(t)-f(t,y^{*}(t))\right| \leq \epsilon \phi(t),~~ t \in \Delta,$$
  where $\phi\in C(\Delta,\mathbb{R}_+)$ is a non-decreasing function such that 
  $$|\mathbf{I}_{a^+}^{\mu;\, \Psi}\phi (t) | \leq \lambda \phi(t), ~t\in \Delta$$ and $\lambda>0$ is a constant satisfying $0<\lambda L<1$.
  Then, there exists a solution $y \in \mathbf{C}_{{1-\rho};\, \Psi}(\Delta,\mathbb{R})$ of equation \eqref{e11} with $\mathbf{I}_{a^+}^{{1-\rho};\, \Psi}y^{*}(a)=\mathbf{I}_{a^+}^{{1-\rho};\, \Psi}y(a) $ such that 
  $$|(\Psi(t)-\Psi(a))^{1-\rho}(y^{*}(t)-y(t))|\leq \left(\frac{\lambda}{1-\lambda L}\,(\Psi(b)-\Psi(a))^{1-\rho}\right)\epsilon\phi(t),~t \in \Delta.$$ 
\end{theorem}
\begin{proof}
For every $\epsilon>0,$ let $ y^{*} \in \mathbf{C}_{{1-\rho};\, \Psi}(\Delta,\mathbb{R})$ satisfies 
  $$\left|^H \mathbf{D}^{\mu,\nu;\,\Psi}_{a^+} y^{*}(t)-f(t,y^{*}(t))\right| \leq \epsilon \phi(t),~ t \in \Delta.$$
Proceeding as in the proof of Theorem \ref{thm4.2} there exists a function $\sigma_{y^{*}} \in \mathbf{C}_{{1-\rho};\, \Psi}(\Delta,\mathbb{R})$ (depending on $y^{*}$) such that
\begin{align*}
y^{*}(t)=\Omega_{\Psi}^{\rho}(t,a)\,\mathbf{I}_{a^+}^{{1-\rho};\, \Psi}y^{*}(a)+\mathbf{I}_{a^+}^{\mu;\, \Psi} f(t, y^{*}(t))
+\mathbf{I}_{a^+}^{\mu;\, \Psi}\sigma_{y^{*}}(t),~t \in \Delta,
\end{align*}
Further, using mathematical induction,  one can prove that the sequence of successive approximations  ${\{y_n\}}_{n=1}^{\infty}\subseteq \mathbf{C}_{{1-\rho};\, \Psi}(\Delta,\mathbb{R})$ defined by
\begin{align}\label{e4.7}
y_n(t)= \Omega_{\Psi}^{\rho}(t,a)\, \mathbf{I}_{a^+}^{{1-\rho};\, \Psi}y^{*}(a)+\frac{1}{\Gamma(\mu)}\int_{a}^{t}\mathcal{L}_{\Psi}^{\mu}(t,\eta)\,  f(\eta, y_{n-1}(\eta))\, d\eta, ~t \in \Delta .
\end{align}
satisfy the inequality
\begin{align} \label{e31}
\| y_j-y_{j-1}\|_{\mathbf{C}_{t;\, \Psi}}\leq \frac{\epsilon}{L}(\lambda L)^j (\Psi(t)-\Psi(a))^{1-\rho}\phi(t),~j\in\N .
\end{align}
Using the inequation \eqref{e31}, we obatin
\begin{align*}
\sum_{j=1}^{\infty}\left\| y_j-y_{j-1}\right\|_{\mathbf{C}_{t;\Psi}} 
\leq \frac{\epsilon}{L} \left( \sum_{j=1}^{\infty}(\lambda L)^{j}\right) (\Psi(t)-\Psi(a))^{1-\rho} \,\phi(t) 
\end{align*}
Thus 
\begin{align}
\sum_{j=1}^{\infty}\left\| y_j-y_{j-1}\right\|_{\mathbf{C}_{t;\Psi}} 
&\leq \epsilon\, \left( \frac{\lambda }{1-\lambda L} \right) (\Psi(t)-\Psi(a))^{1-\rho}  \,\phi(t), ~t \in \Delta. 
\end{align}
Following  the  steps as in the proof of  the Theorem \ref{thm4.2} there exists $y\in \mathbf{C}_{{1-\rho};\, \Psi}(\Delta,\mathbb{R})$ such that $
\| y_n-y\|_{\mathbf{C}_{{1-\rho};\, \Psi}(\Delta,\mathbb{R})}\to 0   ~as ~n\to \infty.$ This $y$ is the solution of the problem \eqref{e11}-\eqref{e12} with 
$\mathbf{I}_{a^+}^{1-\rho,\Psi} y(a)=\mathbf{I}_{a^+}^{1-\rho,\Psi} y^{*}(a)$, and we have 
$$y=y_0+ \sum_{j=1}^{\infty}( y_j-y_{j-1}).$$
Further, for the solution $y^{*}$ of inequation and the solution $y$ of the equation \eqref{e11}, for any $t\in\Delta$, 
\begin{align*} 
|(\Psi(t)-\Psi(a))^{1-\rho}(y^{*}(t)-y(t))|
&=\left|(\Psi(t)-\Psi(a))^{1-\rho}\left[y_0(t)-\left(y_0(t)+\sum_{j=1}^{\infty}( y_j(t)-y_{j-1}(t))\right)\right]\right|\\
&\leq \sum_{j=1}^{\infty}\left|(\Psi(t)-\Psi(a))^{1-\rho}( y_j(t)-y_{j-1}(t))\right|\\
&\leq \sum_{j=1}^{\infty} \| y_j-y_{j-1}\|_{\mathbf{C}_{t, \Psi}}\\
&=\epsilon\, \left( \frac{\lambda }{1-\lambda L} \right) (\Psi(t)-\Psi(a))^{1-\rho}  \,\phi(t), ~ t \in \Delta. 
\end{align*}
Thus, we have
\begin{align*}
|(\Psi(t)-\Psi(a))^{1-\rho}(y^{*}(t)-y(t))|\leq \left(\frac{\lambda}{1-\lambda L}\,(\Psi(b)-\Psi(a))^{1-\rho}\right)\,\epsilon\phi(t), ~ t \in \Delta.
\end{align*}
This proves the equation \eqref{e11} is HUR stable.
\end{proof}
\begin{cor}\label{cor4.2}
Suppose that the function $f$ satisfies the assumptions of Theorem \ref{Thm4.3}. 
Then, the problem \eqref{e11} is generalized HUR stable. 
\end{cor} 
\begin{proof}
Set $\epsilon=1$ and $C_{f,\phi}=\left(\frac{\lambda}{1-\lambda L}\,(\Psi(b)-\Psi(a))^{1-\rho}\right) $ in the proof of Theorem \ref{Thm4.3}.
Then for each solution $y^{*}\in \mathbf{C}_{{1-\rho};\, \Psi}(\Delta,\mathbb{R})$ that satisfies 
the inequality 
$$
\left|^H \mathbf{D}^{\mu,\nu;\,\Psi}_{a^+} y^{*}(t)-f(t,y^{*}(t))\right|\leq \phi(t),~ t \in \Delta,
$$
there exists a solution $y$ of equation \eqref{e11} in $\mathbf{C}_{{1-\rho};\, \Psi}(\Delta,\mathbb{R})$ with $\mathbf{I}_{a^+}^{{1-\rho};\, \Psi}y^{*}(a)=\mathbf{I}_{a^+}^{{1-\rho};\, \Psi}y(a)$ such that 
$$
|(\Psi(t)-\Psi(a))^{1-\rho}(y^{*}(t)-y(t))| \leq C_{f,\phi}\, \phi(t),~t \in \Delta. $$
Hence the fractional differential equation \eqref{e11} is generalized HUR stable.
\end{proof}
\section{$\epsilon-$Approximate solutions to Hilfer FDE}
\begin{definition}
A function $y^{*} \in {\mathbf{C}_{{1-\rho};\, \Psi}(\Delta,\mathbb{R})}$ that satisfy the fractional differential inequality
$$\left|^H \mathbf{D}^{\mu,\nu;\,\Psi}_{a^+}y^{*}(t)-f(t,y^{*}(t))\right| \leq \epsilon,~~ t \in \Delta $$ 
is called an $\epsilon$-approximate solution of $\Psi$--Hilfer FDE \eqref{e11}. 
\end{definition}
\begin{theorem}[\cite{Sousa1}]\label{lem1}
Let $\mathfrak{u},$ $\mathbf{v}$ be two integrable, non negative functions and $\mathbf{g}$ be a continuous, nonnegative, nondecreasing function with domain $\Delta.$  
If
\begin{equation*}
\mathfrak{u}\left( t\right) \leq \mathbf{v}\left( t\right) + \mathbf{g}\left( t\right) \int_{a}^{t}\mathcal{L}_{\Psi}^{\mu}(\tau,s)\mathfrak{u}\left( \tau \right) d\tau,
\end{equation*}
then
\begin{equation}\label{jose}
\mathfrak{u}\left( t\right) \leq \mathbf{v}\left( t\right) +\int_{a}^{t}\overset{\infty }{%
\underset{k=1}{\sum }}\frac{\left[ \mathbf{g}\left( t\right) \Gamma \left( \mu
\right) \right] ^{k}}{\Gamma \left( \mu k\right) }\mathcal{L}_{\Psi}^{\mu k}(t,\tau)\mathbf{v}\left( \tau \right) d\tau,
\end{equation}
$\forall~ t\in \Delta$.
\end{theorem}
\begin{theorem}
Let $f:(a,b] \times \mathbb{R} \to \mathbb{R}$ be a function which satisfies Lipschitz condition
\begin{align*}
 |f(t,y_1)-f(t,y_2)|\leq L |y_1-y_2|
\end{align*}
for each $t\in (a,b]$ and all $y_1,y_2\in \mathbb{R} $, where $L>0$ is constant. Let ${y_i}^{*} \in \mathbf{C}_{{1-\rho};\, \Psi}(\Delta,\mathbb{R}), (i = 1, 2)$ be an $\epsilon_i-$ approximte solutions of FDE \eqref{e11} corresponding to $\mathbf{I}_{a^+}^{1-\rho;\, \Psi}{y_i}^{*}(a)=y_a^{(i)}\in\mathbb{R},$ respectively. Then,  
\begin{small}
\begin{align} \label{dpdn}
\| {y_1}^{*}-{y_2}^{*}\|_{\mathbf{C}_{{1-\rho};\, \Psi}(\Delta,\mathbb{R})}
&\leq (\epsilon_1+\epsilon_2)\left(\frac{(\Psi(b)-\Psi(a))^{\mu-\rho+1}}{\Gamma(\mu+1)}+\sum_{k=1}^{\infty}\frac{L^k}{\Gamma((k+1)\mu-\rho+1)}(\Psi(b)-\Psi(a))^{(k+1)\mu}\right) \nonumber \\
  & \qquad+ |y_a^{(1)}-y_a^{(2)}|\left( \frac{1}{\Gamma(\rho)}+ \sum_{k=1}^{\infty} \frac{L^k}{\Gamma(\rho+k\mu)}
  (\Psi(b)-\Psi(a))^{k\mu}\right).
\end{align}
\end{small}
\end{theorem}
\begin{proof}
Let ${y_i}^{*} \in \mathbf{C}_{{1-\rho};\, \Psi}(\Delta,\mathbb{R}), (i = 1, 2)$ be  an $\epsilon_i-$ approximate solution of FDE \eqref{e11} that satisfy the initial condition $\mathbf{I}_{a^+}^{1-\rho;\, \Psi} {y_i}^{*}(a)=y_a^{(i)}\in\mathbb{R}$. Then, 
\begin{align}\label{ineq5.1}
\left|^H \mathbf{D}^{\mu,\nu;\,\Psi}_{a^+} {y_i}^{*}(t)-f(t,{y_i}^{*}(t))\right|\leq \epsilon_i,~t \in \Delta.
\end{align}
Operating $\mathbf{I}_{a^+}^{\mu;\, \Psi} $ on both the sides  of the above inequation and using the Lemma \ref{lem2.3}, we get
\begin{align*}
\mathbf{I}_{a^+}^{\mu;\, \Psi} {\epsilon_i}
&\geq \mathbf{I}_{a^+}^{\mu;\, \Psi} \left|^H \mathbf{D}^{\mu,\nu;\,\Psi}_{a^+}{y_i}^{*}(t)-f(t,{y_i}^{*}(t))\right|\\
&\geq \left|\frac{1}{\Gamma (\mu)}\int_{a}^{t}\mathcal{L}_{\Psi}^{\mu}(t,\eta) \left(^H \mathbf{D}^{\mu,\nu;\,\Psi}_{a^+}{y_i}^{*}(\eta)-f(t,{y_i}^{*}(\eta))\right)\, d\eta\right|\\
&= \left|\mathbf{I}_{a^+}^{\mu;\, \Psi}\, {^H \mathbf{D}^{\mu,\nu;\,\Psi}_{a^+}} {y_i}^{*}(t)-\mathbf{I}_{a^+}^{\mu;\, \Psi} f(t,{y_i}^{*}(t))\right|\\
&= \left|{y_i}^{*}(t)-\mathbf{I}_{a_+}^{1-\rho; \Psi}{y_i}^{*}(a)\, \Omega_{\Psi}^{\rho}(t,a)- \mathbf{I}_{a^+}^{\mu;\, \Psi} f(t,{y_i}^{*}(t))\right|.
\end{align*}
Therefore,
\begin{align}\label{ineq5.2}
\frac{\epsilon_i}{\Gamma(\mu+1)}\,(\Psi(t)-\Psi(a))^\mu
\geq \left|{y_i}^{*}(t)-y_a^{(i)}\, \Omega_{\Psi}^{\rho}(t,a)-\mathbf{I}_{a^+}^{\mu;\, \Psi} f(t,{y_i}^{*}(t)) \right|,~ i=1,2
\end{align}
Using the following inequalities
$$ |x-y|\leq |x|+|y| ~\mbox{and}~ |x|-|y|\leq |x-y|,~x, y \in \mathbb{R},$$
from the inequation \eqref{ineq5.2}, for any $t\in \Delta,$ we have
\begin{align*}
&\frac{(\epsilon_1+\epsilon_2)}{\Gamma(\mu+1)}\,(\Psi(t)-\Psi(a))^\mu\\
&\geq \left|{y_1}^{*}(t)-y_a^{(1)}\, \Omega_{\Psi}^{\rho}(t,a)-\mathbf{I}_{a^+}^{\mu;\, \Psi} f(t,{y_1}^{*}(t))\right| \\
&+ \left|{y_2}^{*}(t)-y_a^{(2)}\, \Omega_{\Psi}^{\rho}(t,a)-\mathbf{I}_{a^+}^{\mu;\, \Psi} f(t,{y_2}^{*}(t))\right|\\
&\geq \left|\left({y_1}^{*}(t)-y_a^{(1)}\, \Omega_{\Psi}^{\rho}(t,a)-\mathbf{I}_{a^+}^{\mu;\, \Psi} f(t,{y_1}^{*}(t))\right)\right.\\
&\left. \qquad  - \left({y_2}^{*}(t)-y_a^{(2)}\, \Omega_{\Psi}^{\rho}(t,a)-\mathbf{I}_{a^+}^{\mu;\, \Psi} f(t,{y_2}^{*}(t))\right)\right|\\
&=\left|({y_1}^{*}(t)-{y_2}^{*}(t))-(y_a^{(1)}-y_a^{(2)}) \Omega_{\Psi}^{\rho}(t,a) - \mathbf{I}_{a^+}^\mu [f(t,{y_1}^{*}(t))-f(t,{y_2}^{*}(t))]\right|\\
&\geq |({y_1}^{*}(t)-{y_2}^{*}(t))|-\left|(y_a^{(1)}-y_a^{(2)}) \Omega_{\Psi}^{\rho}(t,a)\right| \\
& \qquad- \left|\mathbf{I}_{a^+}^\mu ~\{f(t,{y_1}^{*}(t))-f(t,{y_2}^{*}(t))\}\right|
\end{align*}
Therefore,
\begin{align*}
 |({y_1}^{*}(t)-{y_2}^{*}(t))| 
 & \leq \frac{(\epsilon_1+\epsilon_2)}{\Gamma(\mu+1)}\,(\Psi(t)-\Psi(a))^\mu + \left|(y_a^{(1)}-y_a^{(2)}) \Omega_{\Psi}^{\rho}(t,a)\right|\\
& \qquad + \left| \mathbf{I}_{a^+}^{\mu;\, \Psi} (f(t,{y_1}^{*}(t))-f(t,{y_2}^{*}(t)))\right| \\
& \leq \frac{(\epsilon_1+\epsilon_2)}{\Gamma(\mu+1)}\,(\Psi(t)-\Psi(a))^\mu + \left| (y_a^{(1)}-y_a^{(2)}) \Omega_{\Psi}^{\rho}(t,a)\right| \\
& + \frac{L}{\Gamma(\mu)}\int_{a}^{t}\mathcal{L}_{\Psi}^{\mu}(t,\eta)\,|{y_1}^{*}(\eta))-{y_2}^{*}(\eta)|\, d\eta
\end{align*}
Applaying Lemma \ref{lem1} with 
\begin{align*}
\mathfrak{u}(t)&= |{y_1}^{*}{(t)}-{y_2}^{*}{(t)}|,\\
\mathbf{v}(t)&= \frac{(\epsilon_1+\epsilon_2)}{\Gamma(\mu+1)}\,(\Psi(t)-\Psi(a))^\mu +\left|  (y_a^{(1)}-y_a^{(2)}) \Omega_{\Psi}^{\rho}(t,a)\right| ,\\
\mathbf{g}(t)&=\frac{L}{\Gamma(\mu)},
\end{align*}
we obtain
\begin{align*}
& |{y_1}^{*}{(t)}-{y_2}^{*}{(t)}|\\
 &\leq \frac{(\epsilon_1+\epsilon_2)}{\Gamma(\mu+1)}\,(\Psi(t)-\Psi(a))^\mu + \left|  (y_a^{(1)}-y_a^{(2)}) \Omega_{\Psi}^{\rho}(t,a)\right|\\
 &\quad +\int_{a}^{t}\sum_{k=1}^{\infty}\frac{L^k}{\Gamma(k\mu)}\mathcal{L}_{\Psi}^{k \mu}(t,\eta) \left(\frac{(\epsilon_1+\epsilon_2)}{\Gamma(\mu+1)}\,(\Psi(\eta)-\Psi(a))^\mu + \left|  (y_a^{(1)}-y_a^{(2)}) \Omega_{\Psi}^{\rho}(t,a)\right|\right) \, d\eta \\
 &= \frac{(\epsilon_1+\epsilon_2)}{\Gamma(\mu+1)}\,(\Psi(t)-\Psi(a))^\mu + \left|  (y_a^{(1)}-y_a^{(2)}) \Omega_{\Psi}^{\rho}(t,a)\right|\\
  &\quad +\frac{(\epsilon_1+\epsilon_2)}{\Gamma(\mu+1)}\sum_{k=1}^{\infty} L^k\,\mathbf{I}_{a^+}^{k\mu;\, \Psi}(\Psi(t)-\Psi(a))^\mu+ \frac{|y_a^{(1)}-y_a^{(2)}|}{\Gamma(\rho)}\sum_{k=1}^{\infty} L^k\, \mathbf{I}_{a^+}^{k\mu;\, \Psi}(\Psi(t)-\Psi(a))^{\rho-1} \\
 &= \frac{(\epsilon_1+\epsilon_2)}{\Gamma(\mu+1)}\,(\Psi(t)-\Psi(a))^\mu + \left|  (y_a^{(1)}-y_a^{(2)}) \Omega_{\Psi}^{\rho}(t,a)\right|\\
   &\quad +\frac{(\epsilon_1+\epsilon_2)}{\Gamma(\mu+1)}\,
   \sum_{k=1}^{\infty} L^k\,\frac{\Gamma(\mu+1)}{\Gamma((k+1)\mu+1)}(\Psi(t)-\Psi(a))^{(k+1)\mu}
  \\
     &\quad 
      +\frac{|y_a^{(1)}-y_a^{(2)}|}{\Gamma(\rho)}\sum_{k=1}^{\infty} \frac{L^k\,\Gamma(\rho)}{\Gamma(\rho+k\mu)}(\Psi(t)-\Psi(a))^{k\mu+\rho-1}\\
  &= (\epsilon_1+\epsilon_2)\left(\frac{(\Psi(t)-\Psi(a))^\mu}{\Gamma(\mu+1)}+\sum_{k=1}^{\infty}\frac{L^k}{\Gamma((k+1)\mu+1)}(\Psi(t)-\Psi(a))^{(k+1)\mu}\right) \\
 & + |y_a^{(1)}-y_a^{(2)}|\left( \frac{(\Psi(t)-\Psi(a))^{\rho-1}}{\Gamma(\rho)}+ \sum_{k=1}^{\infty} \frac{L^k}{\Gamma(\rho+ k\mu)}(\Psi(t)-\Psi(a))^{k\mu+\rho-1}\right)
\end{align*}
Thus for every $t\in\Delta,$ we have
\begin{align*}
 & (\Psi(t)-\Psi(a))^{1-\rho}|({y_1}^{*}(t)-{y_2}^{*}(t))| \\
&\leq(\epsilon_1+\epsilon_2)\left(\frac{(\Psi(t)-\Psi(a))^{\mu-\rho+1}}{\Gamma(\mu+1)}+\sum_{k=1}^{\infty}\frac{L^k}{\Gamma((k+1)\mu-\rho+1)}(\Psi(t)-\Psi(a))^{(k+1)\mu}\right) \\
 & + |y_a^{(1)}-y_a^{(2)}|\left( \frac{1}{\Gamma(\rho)}+ \sum_{k=1}^{\infty} \frac{L^k}{\Gamma(\rho+k\mu)}(\Psi(t)-\Psi(a))^{k\mu}\right)\\
 &\leq (\epsilon_1+\epsilon_2)\left(\frac{(\Psi(b)-\Psi(a))^{\mu-\rho+1}}{\Gamma(\mu+1)}+\sum_{k=1}^{\infty}\frac{L^k}{\Gamma((k+1)\mu-\rho+1)}(\Psi(b)-\Psi(a))^{(k+1)\mu}\right) \\
  & + |y_a^{(1)}-y_a^{(2)}|\left( \frac{1}{\Gamma(\rho)}+ \sum_{k=1}^{\infty} \frac{L^k}{\Gamma(\rho+k\mu)}(\Psi(b)-\Psi(a))^{k\mu}\right)\\
\end{align*}
Therefore,
\begin{align*}
&\| {y_1}^{*}-{y_2}^{*}\|_{\mathbf{C}_{{1-\rho};\, \Psi}(\Delta,\mathbb{R})}\\
&\leq (\epsilon_1+\epsilon_2)\left(\frac{(\Psi(b)-\Psi(a))^{\mu-\rho+1}}{\Gamma(\mu+1)}+\sum_{k=1}^{\infty}\frac{L^k}{\Gamma((k+1)\mu-\rho+1)}(\Psi(b)-\Psi(a))^{(k+1)\mu}\right) \\
  & + |y_a^{(1)}-y_a^{(2)}|\left( \frac{1}{\Gamma(\rho)}+ \sum_{k=1}^{\infty} \frac{L^k}{\Gamma(\rho+k\mu)}(\Psi(b)-\Psi(a))^{k\mu}\right)
\end{align*}
which is the desired inequality.
\end{proof}
\begin{rem} \label{rem1}
If  $\epsilon_1=\epsilon_2=0$ in the inequality \eqref{ineq5.1} then ${y_1}^{*}$ and ${y_2}^{*}$ are the solutions of Cauchy problem \eqref{e11}--\eqref{e12} in the space ${\mathbf{C}_{{1-\rho};\, \Psi}[a,b]}$. Further, for  $\epsilon_1=\epsilon_2=0$ the inequality takes the form
\begin{align*} 
\| {y_1}^{*}-{y_2}^{*}\|_{\mathbf{C}_{{1-\rho};\, \Psi}(\Delta,\mathbb{R})}
&\leq |y_a^{(1)}-y_a^{(2)}|\left( \frac{1}{\Gamma(\rho)}+ \sum_{k=1}^{\infty} \frac{L^k}{\Gamma(\rho+k\mu)}
  (\Psi(b)-\Psi(a))^{k\mu}\right),
\end{align*}
 which provides the information  regarding continuous dependence of the solution of the problem \eqref{e11}-\eqref{e12} on initial condition. In addition,  if $y_a^{(1)} = y_a^{(2)}$ we have $\| {y_1}^{*}-{y_2}^{*}\|_{\mathbf{C}_{{1-\rho};\, \Psi}(\Delta,\mathbb{R})}=0$, which gives the uniqueness of solution of the problem \eqref{e11}-\eqref{e12}.
\end{rem}
\section{Examples}
\begin{ex} 
Consider the $\Psi$--Hilfer FDE
\begin{align}
^H \mathbf{D}_{0^+}^{\frac{1}{2},\frac{1}{2};\, \Psi} y(t)&=4 y(t)  \label{ex1} , ~t \in J=[0,1],\\
\mathbf{I}_{0^+}^{\frac{1}{4}}y(0)&=2. \label{ex12} 
\end{align}
\end{ex}
comparing with the Cauchy problem \eqref{e11}-\eqref{e12}, we have
$$\mu=\frac{1}{2},\, \nu=\frac{1}{2},\, \rho= \mu+\nu-\mu\nu=\frac{3}{4},\,\, {y}^{*}_0= \mathbf{I}_{0^+}^{1-\rho;\, \Psi}{y}^{*}(0)=2,\,\mbox{and} ~f(t,y(t))=4y(t).$$
Clearly, $f$ satisfies Lipschitz condition with Lipschitz constant $L=4$. By Theorem \ref{th3.1} the initial value problem \eqref{ex1}-\eqref{ex12} has a unique solution. Further, the Theorem \ref{thm4.2} guarantee that the equation \eqref{ex1} is HU stable.
Indeed, we prove that for given $\epsilon>0$ and the solution ${y}^{*}$ of the inequality
\begin{align*}
\left|^H \mathbf{D}_{0^+}^{\frac{1}{2},\frac{1}{2};\, \Psi}{y}^{*}(t)-4 {y}^{*}(t)\right|\leq \epsilon,~t\in [0,1],
\end{align*} 
we can find a constant $C$ and solution $y$ of the given equation \eqref{ex1} such that
$$\|{y}^{*}-y\|_{\mathbf{C}_{1-\rho;\, \Psi}} \leq C \, \epsilon.$$
For example, take $\epsilon=8$ and consider the inequality
\begin{align}\label{61}
\left|^H \mathbf{D}_{0^+}^{\frac{1}{2},\frac{1}{2};\, \Psi}{y}^{*}(t)-4 {y}^{*}(t)\right|\leq 8,~t\in [0,1].
\end{align} 
Note that the function $\tilde{{y}^{*}}(t)=2 \frac{(\Psi(t)-\Psi(0))^{-\frac{1}{4}}}{\Gamma{(\frac{3}{4})}}\, $ satisfies the inequality \eqref{61}. Further,
$$
^H \mathbf{D}_{0^+}^{\frac{1}{2},\frac{1}{2};\, \Psi} {y_1}^{*}(t)=0,
$$
which shows $\tilde{{y}^{*}}$ is not the solution of  the Cauchy problem (\ref{ex1}) - (\ref{ex12}). Next, as discussed in the proof of Theorem \ref{thm4.2}, we define the sequence of successive approximations to the solution of (\ref{ex1}) as follows:
\begin{align*}
y_0(t)&= \tilde{{y}^{*}}(t)=2 \frac{(\Psi(t)-\Psi(0))^{-\frac{1}{4}}}{\Gamma{(\frac{3}{4})}}\\
y_n(t)&= \Omega_{\Psi}^{\rho}(t,a)\, \mathbf{I}_{a^+}^{{1-\rho};\, \Psi}{y}^{*}(a)+\frac{1}{\Gamma(\mu)}\int_{a}^{t}\mathcal{L}_{\Psi}^{\mu}(t,\eta)\, f(\eta, y_{n-1}(\eta))\, d\eta\\
&=\frac{(\Psi(t)-\Psi(0))^{-\frac{1}{4}}}{\Gamma(\frac{3}{4})}\,+\frac{4}{\Gamma(\frac{1}{2})}\int_{0}^{t}\mathcal{L}_{\Psi}^{\frac{1}{2}}(t,\eta) \,y_{n-1}(\eta))\, d\eta, ~t \in J, n \in \mathbb{N}. 
\end{align*}
Then,
\begin{align*}
y_1(t)&= 2 \frac{(\Psi(t)-\Psi(0))^{-\frac{1}{4}}}{\Gamma(\frac{3}{4})}\,+8\frac{(\Psi(t)-\Psi(0))^{\frac{1}{4}}}{\Gamma(\frac{5}{4})}\,\\
y_2(t)&= 2(\Psi(t)-\Psi(0))^{-\frac{1}{4}}\left[\frac{1}{\Gamma(\frac{3}{4})}+4 \frac{(\Psi(t)-\Psi(0))^{\frac{1}{2}}}{\Gamma(\frac{5}{4})}\,+ \frac{16}{\Gamma(\frac{7}{4})}\,(\Psi(t)-\Psi(0)) \right].
\end{align*}
In general, we have
\begin{align*}
y_n(t)&= 2(\Psi(t)-\Psi(0))^{-\frac{1}{4}} \sum_{j=0}^{n}\frac{(4(\Psi(t)-\Psi(0))^{\frac{1}{2}})^j}{\Gamma({j\frac{1}{2}}+\frac{3}{4})}, n \in \mathbb{N}. 
\end{align*}
The exact solution of the initial value problem \eqref{ex1}-\eqref{ex12} is given by
\begin{align}\label{e6.2}
y(t)&=\lim_{n\to\infty}{y_n(t)} \nonumber \\
& = \lim_{n\to \infty}2(\Psi(t)-\Psi(0))^{-\frac{1}{4}} \sum_{j=0}^{n}\frac{(4(\Psi(t)-\Psi(0))^{\frac{1}{2}})^j}{\Gamma({j\frac{1}{2}}+\frac{3}{4})} \nonumber \\
&= 2(\Psi(t)-\Psi(0))^{-\frac{1}{4}}\, E_{\frac{1}{2},\frac{3}{4}}(4 (\Psi(t)-\Psi(0))^{\frac{1}{2}}) 
\end{align}
Therefore
\begin{align*}
\|{y}^{*}-y\|_{\mathbf{C}_{1-\rho;\, \Psi}}&= \max_{t\in[0,1]} \left|(\Psi(t)-\Psi(0))^{1-\rho} ({y}^{*}(t)-y(t))\right|\\
&= \max_{t\in[0,1]}\left|(\Psi(t)-\Psi(0))^{\frac{1}{4}} (y(t)-x(t))\right|\\
&\leq \max_{t\in[0,1]} \left|{y}^{*}(t)-y(t)\right|\\
&\leq \max_{t\in[0,1]} \left|2(\Psi(t)-\Psi(0))^{-\frac{1}{4}}\, E_{\frac{1}{2},\frac{3}{4}}(4 t^{\frac{1}{2}})-\frac{2}{\Gamma(\frac{3}{4})}\,(\Psi(t)-\Psi(0))^{-\frac{1}{4}}\right|\\
&\leq \max_{t\in[0,1]} \left|E_{\frac{1}{2},\frac{3}{4}}(4 (\Psi(t)-\Psi(0))^{\frac{1}{2}})-\frac{2}{\Gamma(\frac{3}{4})}\right|\\
&\leq \left|E_{\frac{1}{2},\frac{3}{4}}(4 (\Psi(1)-\Psi(0))^{\frac{1}{2}})-\frac{2}{\Gamma(\frac{3}{4})}\right|\\
&=  C_{f} \,\epsilon,
\end{align*}
where $C_{f} = \frac{1}{8}\, \left|E_{\frac{1}{2},\frac{3}{4}}(4 (\Psi(1)-\Psi(0))^{\frac{1}{2}})-\frac{2}{\Gamma(\frac{3}{4})}\right|.$
On the similar line, for each $\epsilon>0$ and  for each solution
${y}^{*}\in \mathbf{C}_{{1-\rho};\, \Psi}[a,b]$ of the inequation \eqref{ine1}, one can find by method of successive approximation  a solution  $y \in \mathbf{C}_{{1-\rho};\, \Psi}[a,b]$ of \eqref{ex1} that satisfies the inequality
$$\|{y}^{*}-y\|_{\mathbf{C}_{{1-\rho};\, \Psi}[a,b]}\leq C_{f} \,\epsilon. $$

% % % % % % % % % % % % % % % % % % % % % % % % % % % % % % %

\end{document}